\documentclass{article}

\usepackage{amsmath}
\usepackage{amsthm}
\usepackage{amssymb}
\usepackage{latexsym}
\usepackage{enumerate}

\newtheorem{thm}{Theorem}[section]
\newtheorem{cor}[thm]{Corollary}
\newtheorem{lem}[thm]{Lemma}

\theoremstyle{definition}
\newtheorem{defin}[thm]{Definition}
\newtheorem{rem}[thm]{Remark}

\numberwithin{equation}{section}

\newtheorem{notation}[thm]{Notation}

\newcommand{\Z}{\mathbb{Z}}
\newcommand{\Q}{\mathbb{Q}}

\newcommand{\GL}{{\rm GL}}

\newcommand{\Ccal}{\mathcal{C}}

\begin{document}

\title{A characterization of B\"uchi's integer sequences of length $3$}

\author{Pablo S\'aez\\
Computer Science and IT Department\\
Universidad del B\'io B\'io\\
Chill\'an, Chile\\
E-mail: psaezg@ubiobio.cl
\and 
Xavier Vidaux\\
Departamento de Matem\'atica\\ 
Facultad de Ciencias F\'isicas y Matem\'aticas\\
Casilla 160 C\\
Universidad de Concepci\'on\\
E-mail: xvidaux@udec.cl}

\date{}

\maketitle

\renewcommand{\thefootnote}{}



\footnote{This research was partially supported by the second author's Chilean research project FONDECYT 1090233.}

\renewcommand{\thefootnote}{\arabic{footnote}}

\setcounter{footnote}{0}


\begin{abstract}
We give a new characterization of generalized B\"uchi sequences (sequences whose sequence of squares has constant second difference $(a)$, for some fixed integer $a$) of length $3$ over the integers and a strategy for attacking B\"uchi's $n$ Squares Problem. Known characterizations of integer B\"uchi sequences of length 3 are actually characterizations over $\Q$, plus some divisibility criterions that keep integer sequences. 
\end{abstract}


\section{Introduction and Notation}

A \emph{B\"uchi sequence} over a commutative ring $A$ with unit is a sequence of 
elements of $A$ whose second difference of squares is the constant sequence $(2)$ 
(e.g. $(0,7,10)$ is a B\"uchi sequence). 
Since the first difference of a sequence of consecutive squares, e.g. $(4,9,16,25)$, is 
a sequence of consecutive odd numbers - in the example $(5,7,9)$ - the second difference 
of such a sequence is the constant sequence $(2)$. A B\"uchi sequence $(x_n)$ is called \emph{trivial} if there exists $x\in A$ such that for all $n$ we have $x_n^2=(x+n)^2$ (e.g. $(-2,3,4)$ and $(-4,3,2)$). Note that a B\"uchi sequence $(x_1,x_2,x_3)$ of integers satisfies
\begin{eqnarray}\label{buchi}
x_3^2-2x_2^2+x_1^2=2.
\end{eqnarray}

The largest known non-trivial B\"uchi sequences over the integers have length $4$ 
(infinitely many such sequences are known - see for example \cite{Hensley2} 
or \cite{PastenPheidasVidaux}). \emph{B\"uchi's Problem} over a commutative ring $A$ with 
unit asks whether there exists an integer $M$ such that no non-trivial B\"uchi 
sequence of length $\ge M$ exists in $A$. B\"uchi's Problem over the integers is open. Although this problem had been studied by B\"uchi himself in the early seventies (or maybe even in the sixties), it became known to the general mathematical community only after being publicized by Lipshitz \cite{Lipshitz} in 1990. Two very interesting papers on this problem by D. Hensley \cite{Hensley, Hensley2} from the early eighties' were unfortunately never published. 

Though it is a very natural problem of Arithmetic, it seems that the main motivation of B\"uchi resided in Mathematical Logic. Indeed, he observed that if this problem had a positive answer, then using the fact that the positive existential theory of $\Z$ in the language of rings is undecidable (a consequence of the negative answer to Hilbert's Tenth Problem by Matiyasevic, after works by M. Davis, H. Putnam and J. Robinson - see for example \cite{Matiyasevic} or \cite{Davis}), he could prove that the problem of simultaneous representation of integers by a system of diagonal quadratic forms over $\Z$ would be undecidable (see \cite{PastenPheidasVidaux} for a more general discussion about this aspect of B\"uchi's Problem). 

There are various evidences that B\"uchi's Problem would have a positive answer over 
the rational numbers (hence also over the integers). First in 1980, Hensley 
\cite{Hensley} gave a heuristic proof using counting arguments. In 2001, P. Vojta \cite{Vojta2} gave a proof (that works actually over any number field) that depends on a conjecture by Bombieri about the locus of rational points on projective varieties of general type over a number field, giving at the same time a geometric motivation for solving B\"uchi's Problem. In 2009, H. Pasten proved, following Vojta, that a strong version of B\"uchi's Problem would have a positive answer over any number field if Bombieri's conjecture had a positive answer for surfaces - see \cite{Pasten2}. 

For other results related to B\"uchi's Problem, we refer to \cite{PastenPheidasVidaux} and \cite{BrowkinBrzezinski}. 

Consider a B\"uchi sequence $(x_1,x_2,x_3)$ over $\Q$, i.e. a sequence satisfying Equation \eqref{buchi}, and write $x_2=x_1+u$ and $x_3=x_1+v$. Equation \eqref{buchi} becomes 
$$
(x_1+v)^2-2(x_1+u)^2+x_1^2=2
$$
hence 
$$
2vx_1+v^2-4ux_1-2u^2=2. 
$$
So we can write $x_1$, $x_2$ and $x_3$ as rational functions of the variables $u$ and $v$ such that for any rational numbers $u$ and $v$, the sequence 
$$
(x_1(u,v),x_2(u,v),x_3(u,v))
$$ 
is a B\"uchi sequence over $\Q$. Writing $x_2=x_1+u+v$ and $x_3=x_1+u+2v$ and applying the same method as above, Hensley \cite{Hensley2} obtains a parametrization a bit simpler that allows him to show that the sequences $(x_1,x_2,x_3)$ over $\Z$ which satisfy $0\le x_1<x_2<x_3$ are characterized by the above parametrization by adding the conditions that $u$ and $v$ are both integers and, $u$ is even and divides $v^2-1$. Note that the ``missing'' sequences are then obtained by taking all the symmetric sequences $(x_3,x_2,x_1)$ and adding some minus signs randomly in front of the $x_i$'s.

In this paper, we produce a direct characterization of \emph{generalized} B\"uchi sequences of length $3$ over the integers (solutions over $\Z$ to the equation $x_3^2-2x_2^2+x_1^2=a$, where $a$ is any fixed integer), and propose a strategy for solving B\"uchi's Problem. 

In order to state our theorems, we need first to introduce some notation.

\begin{notation}
\begin{itemize}
\item For any integer $a$, we will denote by $\Gamma_a$ the set of integer solutions of Equation
\begin{equation}\label{eqBa}
x_1^2-2x_2^2+x_3^2=a
\end{equation}
and by $\Omega_a$ the set of integer solutions of Equation
\begin{equation}\label{eqOa}
-2x_1^2+x_2^2-2x_3^2=a.
\end{equation}
We will often abuse notation by identifying elements $x=(x_1,x_2,x_3)$ of $\Gamma_a$ with the corresponding column matrix and elements $x=(x_1,x_2,x_3)$ of $\Omega_a$ with the row matrix $\begin{pmatrix}x_1&x_2&x_3\end{pmatrix}$.
\item Let 
$$
B=\begin{pmatrix}
3&4&0\\
2&3&0\\
0&0&1
\end{pmatrix}
\qquad\textrm{and}\qquad
J=\begin{pmatrix}
0&0&1\\
0&1&0\\
1&0&0
\end{pmatrix}.
$$
Note that $B$ has determinant $1$ and $J$ has determinant $-1$. Indeed we have $J^{-1}=J$ and
$$
B^{-1}=\begin{pmatrix}
3&-4&0\\
-2&3&0\\
0&0&1
\end{pmatrix}.
$$
\item Let $H=<B,J>$ be the subgroup of $\GL_3(\Z)$ generated by $B$ and $J$.
\item Write $\Ccal=\{x\in\Z^3\colon |x_1|\leq |x_2| \textrm{ or } |x_1|\geq2|x_2|\}$.
\item Let  
$$
\Theta_a=
\begin{cases}
\{(x_1,x_2,x_3)\in\Gamma_a\colon |x_2|\geq\max\{|x_1|,|x_3|\}\}&\textrm{if } a<0\\
\{x\in\Gamma_a\colon x\in\Ccal \textrm{ and }Jx\in\Ccal\}&\textrm{if } a\geq 0
\end{cases}
$$
and note that for any $x\in\Theta_a$, also $Jx\in\Theta_a$ (the equation defining $\Gamma_a$ is symmetric in $x_1$ and $x_3$). Note also that each $\Theta_a$ is a subset of $\Gamma_a$. 
\item Let 
$$
\Delta_2=\left\{(2,1,0),(-2,1,0),(1,0,1),(-1,0,1),(-1,0,-1)\right\}
$$
and note that $\Delta_2$ is a subset of $\Theta_2$.
\item Let $\Delta'_{-2}=\{(1,0,0),(-1,0,0)\}$ and $\Delta'_{1}=\{(0,1,0),(0,-1,0)\}$, and note that for each $a\in\{1,-2\}$, the set $\Delta'_a$ is a subset of $\Omega_a$.
\end{itemize}
\end{notation}

The following theorem, proved in Section \ref{sec:act}, consists essentially of observations, but it contains the initial ideas for this paper. The idea of using the matrix $B$ comes from the solution of Problem 204 in Sierpi\'nski \cite{Sierpinski}. 

\begin{thm}\label{act}
The group $H$ acts on each $\Gamma_a$ by left multiplication and it acts on each $\Omega_a$ by right multiplication (in particular, the orbit of each $\Theta_a$ is included in $\Gamma_a$). Moreover, if $M\in H$ then 
\begin{enumerate}
\item the first and third columns of $M$ belong to $\Gamma_{1}$ and the second column of $M$ belongs to $\Gamma_{-2}$; and
\item the first and third rows of $M$ belong to $\Omega_{-2}$ and the second row of $M$ belongs to $\Omega_{1}$.
\end{enumerate}
\end{thm}

We want to find for each integer $a$ a set as small as possible, finite if possible, whose orbit through the action of $H$ is exactly the set $\Gamma_a$. Next two theorems, proved in Sections \ref{sec:fin} and \ref{sec:main} respectively, tell us that the sets $\Theta_a$ are good candidates. 

\begin{thm}\label{teo:fin}
For each $a\ne0$ the set $\Theta_a$ is finite. In particular, we have
$$
\Theta_{-2}=\{(0,\pm1,0)\}\quad\Theta_{-1}=\{(\pm1,\pm1,0),(0,\pm1,\pm1)\}
$$
$$
\Theta_0=\{(x_1,x_2,x_3)\in\Z^3\colon |x_1|=|x_2|=|x_3|\}
$$
$$
\Theta_1=\{(\pm1,0,0),(0,0,\pm1)\}\quad\Theta_2=\{(\pm 2,\pm 1,0),(0,\pm 1,\pm2),(\pm 1,0,\pm 1)\}
$$
where the $\pm$ signs are independent (so for example $\Theta_2$ has $12$ elements).  
\end{thm}

\begin{thm}\label{main}
For each integer $a$ the orbit of $\Theta_a$ is $\Gamma_a$. 
\end{thm}

There is some obvious (possible) redundancy in each set $\Theta_a$: for example, for each $x\in\Theta_a$ such that $x\ne Jx$, we could take one of $x$ or $Jx$ out of the set. We were not able to find an optimal subset of $\Theta_a$ for each $a$ (in a uniform way), but when $a=2$, it is not hard to see that the set $\Delta_2$ defined above is actually enough to generate all the sequences in $\Theta_2$, so that we have indeed (proved in Section \ref{sec:main}): 

\begin{cor}\label{cor:Delta2}
The orbit of $\Delta_2$ is $\Gamma_2$. 
\end{cor}

In Section \ref{sec:Prel} we will prove a series of lemmas that will allow us to show, in particular, the two following theorems in Sections \ref{sec:pres} and \ref{sec:main2} respectively. 

\begin{thm}\label{pres}
The group $H$ has presentation $\langle x,y \mid y^2\rangle$, hence it is isomorphic to the free product $\Z*\Z_2$. 
\end{thm}

\begin{thm}\label{main2}
Given a $3$-terms B\"uchi sequence $x=(x_1,x_2,x_3)$ of integers there exists a matrix $M\in H$ and a unique $\delta\in\Delta_2$ such that $x=M\delta$. Moreover, the matrix $M$ is unique with this property if $\delta\notin\{(1,0,1),(-1,0,-1)\}$, and it is unique up to right-multiplication by $J$ otherwise. 
\end{thm}

The existence part of Theorem \ref{main2} is just Corollary \ref{cor:Delta2}. The fact that $\Delta_2$ is somewhat \emph{optimal} comes from the unicity part. In particular, there are exactly five orbits, and we show in Section \ref{sec:orb} that in order to know in what orbit a sequence $(x_1,x_2,x_3)$ lies, it is enough to know the residues of $x_1$ and of $x_3$ modulo $8$ (see Theorem \ref{thm:orb}). 

In Section \ref{sec:strat}, we will describe a general strategy for trying to show that all B\"uchi sequences of length $5$ are trivial, and another strategy, that seems to be more promising, for trying to show that all B\"uchi sequences of length $8$ are trivial. 

J. Browkin suggested to us the reference \cite[Section 13.5, p. 301]{Cassels} as it is explained how to characterize integer solutions of isotropic ternary forms through a very specific action of a subgroup of $\GL_2(\Q)$. This approach has the advantage of dealing with groups that are better known than our group $H$, but the action itself is much less natural than ours, and it is not clear to us which of the two approaches would give a better insight into B\"uchi's problem. For example, the characterization of the orbits seems harder with Cassel's approach.

\section{Proof of Theorem \ref{act}}\label{sec:act}

Choose an arbitrary $x=(x_1,x_2,x_3)\in\Z^3$. On the one hand, the sequence $Jx=(x_3,x_2,x_1)$ (respectively $xJ$) is an element of $\Gamma_a$ (respectively $\Omega_a$) if and only if $x\in\Gamma_a$ (respectively $x\in\Omega_a$), since Equations \eqref{eqBa} and \eqref{eqOa} are symmetric in $x_1$ and $x_3$. Moreover, we have\,:
$$
Bx=\begin{pmatrix}3x_1+4x_2\\2x_1+3x_2\\x_3\end{pmatrix}
\qquad\textrm{and}\qquad
xB=\begin{pmatrix}3x_1+2x_2&4x_1+3x_2&x_3\end{pmatrix}
$$
and we have 
$$
x_3^2-2(2x_1+3x_2)^2+(3x_1+4x_2)^2=x_3^2-2x_2^2+x_1^2
$$
and 
$$
-2x_3^2+(3x_1+2x_2)^2-2(4x_1+3x_2)^2=-2x_3^2+x_2^2-2x_1^2.
$$
Hence $Bx$ satisfies Equation \eqref{eqBa} if and only if $x$ satisfies it, and $xB$ satisfies Equation \eqref{eqOa} if and only if $x$ satisfies it. Since $J$ and $B$ are in $\GL(3,\Z)$, we can conclude that $H$ acts on $\Gamma_a$ and $\Omega_a$. 

Let $M$ be a matrix in $H$ with columns $c_1$, $c_2$ and $c_3$ and with rows $r_1$, $r_2$ and $r_3$. Since 
$$
M=(c_1,c_2,c_3)=\left(M\begin{pmatrix}1\\0\\0\end{pmatrix},M\begin{pmatrix}0\\1\\0\end{pmatrix},M\begin{pmatrix}0\\0\\1\end{pmatrix}\right)
$$
the columns $c_1$ and $c_3$ are in the orbit of $\Delta_1\subseteq\Theta_1$, hence are in $\Gamma_1$, and $c_2$ is in the orbit of $\Delta_{-2}\subseteq\Theta_{-2}$, hence is in $\Gamma_{-2}$. Since 
$$
M=\begin{pmatrix}r_1\\r_2\\r_3\end{pmatrix}=
\begin{pmatrix}\begin{pmatrix}1&0&0\end{pmatrix}M\\\begin{pmatrix}0&1&0\end{pmatrix}M\\\begin{pmatrix}0&0&1\end{pmatrix}M\end{pmatrix}
$$
the rows $r_1$ and $r_3$ are in the orbit of $\Delta'_{-2}$, hence are in $\Omega_{-2}$, and $r_2$ is in the orbit of $\Delta'_{1}$, hence is in $\Omega_{1}$.

\section{Proof of Theorem \ref{teo:fin}}\label{sec:fin}

We separate the cases $a\geq0$ and $a<0$. \\

\textbf{CASE $a\geq0$.} If $x\in\Theta_a$ then $x\in\Ccal$ and $Jx\in\Ccal$, hence we have four Cases:
\begin{enumerate}
\item $|x_1|\leq |x_2|$ and $|x_3|\leq |x_2|$
\item $|x_1|\leq |x_2|$ and $|x_3|\geq 2|x_2|$
\item $|x_1|\geq 2|x_2|$ and $|x_3|\leq |x_2|$
\item $|x_1|\geq 2|x_2|$ and $|x_3|\geq 2|x_2|$
\end{enumerate}

\textbf{Case 1:} We have $x_1^2-2x_2^2+x_3^2\leq0$, so that Equation \eqref{eqBa} has no solution at all in this case, unless $a=0$. If $a=0$ and, either $|x_1|\ne|x_2|$ or $|x_3|\ne|x_2|$, then $0=x_1^2-2x_2^2+x_3^2<0$, which is absurd. Hence if $a=0$ then $|x_1|=|x_2|=|x_3|$.

\textbf{Cases 3 and 4:} If $|x_1|\geq 2|x_2|$ then $2x_2^2+x_3^2\leq x_1^2-2x_2^2+x_3^2\leq a$ and there are only finitely many sequences that satisfy 
\begin{equation}\label{eq:thet34}
2x_2^2+x_3^2\leq a.
\end{equation}
If $a=0$ then the only solution is $x_1=x_2=x_3=0$. Let us now find the exact solutions when $a=1$ or $a=2$. \\
\emph{Subcase (i): $|x_3|\leq|x_2|$.} Equation \eqref{eq:thet34} gives then $3x_3^2\leq a$, hence $x_3=0$ and Equation \eqref{eq:thet34} becomes $2x_2^2\leq a$. So in the case that $a=1$, we find $x_2=0$ and $1=x_1^2-2x_2^2+x_3^2=x_1^2$, which give the solutions $(\pm1,0,0)$. In the case that $a=2$, we find that $x_2^2$ can be $0$ or $1$, but since $x_1^2-2x_2^2=2$ (by definition of $\Gamma_2$) we deduce that $x_2^2=1$, hence $x_1^2=4$, which gives the solutions $(\pm2,\pm1,0)$. \\
\emph{Subcase (ii): $|x_3|\geq2|x_2|$.} Equation \eqref{eq:thet34} gives then $6x_2^2\leq a$, hence $x_2=0$ and Equation \eqref{eq:thet34} becomes $x_3^2\leq a$, so $x_3^2\leq1$. In the case that $a=1$, we have $1=x_1^2+x_3^2$ (by definition of $\Gamma_1$) hence the solutions are of the form $(\pm1,0,0)$ or $(0,0,\pm1)$. In the case that $a=2$, we have $2=x_1^2+x_3^2$, hence $x_1^2=x_3^2=1$ and the solutions are $(\pm1,0,\pm1)$. 

\textbf{Case 2:} Since the equation defining $\Gamma_a$ is symmetric in $x_1^2$ and $x_3^2$, we deduce from the study of Case 3 that there are only finitely many sequences and that if $a=0$ then the only solution is $(0,0,0)$. Again by symmetry, the study of Subcase (ii) of Case 3 tells us that if $a=1$ then the solutions are of the form $(0,0,\pm1)$, and if $a=2$ then the solutions are of the form $(0,\pm1,\pm2)$.\\

\textbf{CASE $a<0$.} In this case, we have $|x_2|\geq\max\{|x_1|,|x_3|\}$, hence $x_2^2-x_1^2$ and $x_2^2-x_3^2$ are non-negative integers. Since
$$
0<-a=-x_1^2+2x_2^2-x_3^2=(x_2^2-x_1^2)+(x_2^2-x_3^2)
$$
we conclude that there are only finitely many choices for $x_2^2-x_1^2$ and $x_2^2-x_3^2$. For each such choice, there are only finitely many choices for each $x_i$ since $x_2^2-x_i^2=(x_2-x_i)(x_2+x_i)$. In the case that $a=-1$, we have $x_2^2-x_1^2=1$ and $x_2^2-x_3^2=0$, which gives the solutions $(0,\pm1,\pm1)$, or the symmetric case that gives the solutions $(0,\pm1,\pm1)$. Assume now that $a=-2$. Since a difference of two squares cannot be $2$, we have $x_2^2-x_1^2=1$ and $x_2^2-x_3^2=1$, which gives the solutions $(0,\pm1,0)$.

\section{Proof of Theorem \ref{main}}\label{sec:main}

The idea is to define a function $\varphi\colon\Gamma_a\rightarrow\Gamma_a$ constant on $\Theta_a$, involving only $J$, $B$ and $B^{-1}$, such that, given $x=(x_1,x_2,x_3)\in\Gamma_a$ there exists a positive integer $n$ depending only on $x$ such that the $n$-th iterate $\varphi^n(x)$ belongs to $\Theta_a$ (where $\varphi^n$ denotes the function $\varphi$ composed $n$ times with itself). 

Recalling that $\Ccal=\{x\in\Z^3\colon |x_1|\leq |x_2| \textrm{ or } |x_1|\geq2|x_2|\}$, the following four sets 
$$
\begin{aligned}
\Theta_a&\\
\Gamma_a^1&=\{x\in\Gamma_a\smallsetminus\Theta_a\colon x\notin\Ccal \textrm{ and } x_1x_2>0\}\\
\Gamma_a^0&=\{x\in\Gamma_a\smallsetminus\Theta_a\colon x\in\Ccal\}\\
\Gamma_a^{-1}&=\{x\in\Gamma_a\smallsetminus\Theta_a\colon x\notin\Ccal \textrm{ and } x_1x_2<0\}
\end{aligned}
$$
form a partition of $\Gamma_a$ (observe that if $x_1$ or $x_2$ is $0$ then $x$ is in $\Gamma_a^0$).

The function $\varphi_a$ is defined in the following way\,:
$$
\varphi_a(x)=
\begin{cases}
x\quad&\textrm{if $x\in\Theta_a$}\\
Jx&\textrm{if $x\in\Gamma_a^0$}\\
B^{-1}x&\textrm{if $x\in\Gamma_a^1$}\\
Bx&\textrm{if $x\in\Gamma_a^{-1}$}.
\end{cases}
$$

\begin{notation}
If $x=(x_1,x_2,x_3)\in\Gamma_a$ then we will write 
$$
\varphi_a(x)=(\varphi_a(x)_1,\varphi_a(x)_2,\varphi_a(x)_3).
$$
\end{notation}

The following lemma finishes the proof of the Theorem. 

\begin{lem} Let $x=(x_1,x_2,x_3)\in\Gamma_a$. We have\,:
\begin{enumerate}
\item $\Theta_a$ is fixed by $\varphi_a$;
\item $\varphi_a(\Gamma_a^0)\subseteq\Theta_a\cup\Gamma_a^1\cup\Gamma_a^{-1}$, and if $x\in\Gamma_a^0$ then $\varphi_a(x)_2=x_2$; and
\item if $x\in\Gamma_a^1\cup\Gamma_a^{-1}$ then $|\varphi_a(x)_2|<|x_2|$.
\end{enumerate}
Therefore, for all $x\in\Gamma_a$ there exists a positive integer $n$ such that\,: for all integer $m\geq n$ we have $\varphi_a^m(x)=\varphi_a^n(x)\in\Theta_a$.
\end{lem}
\begin{proof}
Assume that the three items have been proven and let $x\in\Gamma_a\smallsetminus\Theta_a$. Applying Items 2 and 3 of the lemma repeatedly, the second term of the sequence decreases in absolute value until getting to an element of $\Theta_a$, and the conclusion of the lemma follows.

Let us now prove each item. 
\begin{enumerate}
\item By definition of $\varphi_a$.
\item If $x=(x_1,x_2,x_3)\in\Gamma_a^0$ then $\varphi_a(x)=Jx=(x_3,x_2,x_1)$, hence trivially $\varphi_a(x)_2=x_2$. In order to obtain a contradiction, suppose $\varphi_a(x)\in\Gamma_a^0$, so that we have: $x\in\Ccal$ and $Jx\in\Ccal$. If $a\geq0$, this means that $x\in\Theta_a$, which contradicts the hypothesis on $x$. So we may suppose $a<0$. We have four cases:
\begin{enumerate}
\item $|x_1|\leq |x_2|$ and $|x_3|\leq |x_2|$
\item $|x_1|\leq |x_2|$ and $|x_3|\geq 2|x_2|$
\item $|x_1|\geq 2|x_2|$ and $|x_3|\leq |x_2|$
\item $|x_1|\geq 2|x_2|$ and $|x_3|\geq 2|x_2|$
\end{enumerate}
Case (a) is impossible since otherwise $x$ would be in $\Theta_a$. If $|x_1|\geq|2x_2|$ then by Equation \eqref{eqBa}, we have 
$$
0>a=x_1^2-2x_2^2+x_3^2\geq2x_2^2+x_3^2\geq0
$$
which is impossible. The cases where $|x_3|\geq2|x_2|$ are done similarly. 

\item Let $\varepsilon\in\{-1,1\}$ and suppose $x\in\Gamma_a^\varepsilon$. We have 
$$
\varphi_a(x)=B^{-\varepsilon}x=
\begin{pmatrix}
3&-4\varepsilon &0\\
-2\varepsilon &3&0\\
0&0&1
\end{pmatrix}
\begin{pmatrix}x_1\\x_2\\x_3\end{pmatrix}=
\begin{pmatrix}3x_1-4\varepsilon x_2\\-2\varepsilon x_1+3x_2\\x_3\end{pmatrix}
$$
hence 
\begin{equation}\label{eq1}
\varphi_a(x)_2=-2\varepsilon x_1+3x_2=x_2+2(-\varepsilon x_1+ x_2).
\end{equation} 

Note that by definition of $\Gamma_a^\varepsilon$ we have $\varepsilon x_1x_2>0$, hence 
in particular $x_2\ne0$ and we need only consider the case where $x_2$ is positive and the case where $x_2$ is negative.

\noindent\textbf{Case 1\,.} \emph{$x_2$ is positive.} By definition of $\Gamma_a^\varepsilon$, $\varepsilon$ and $x_1$ have the same sign. Since $|x_1|>|x_2|$ (by definition of $\Gamma_a^\varepsilon$), we have 
$$
-\varepsilon x_1+ x_2=-|x_1|+ x_2<0.
$$ 
Hence by Equation \eqref{eq1}, we have 
$\varphi_a(x)_2<x_2$. On the other hand, we have $2x_2>|x_1|$, hence 
$$
\varphi_a(x)_2=-2\varepsilon x_1+3x_2=-2|x_1|+3x_2>-x_2
$$
and we conclude $|\varphi_a(x)_2|<|x_2|$. 

\noindent\textbf{Case 2\,.} \emph{$x_2$ is negative.} This case is done similarly and is left to the reader. 
\end{enumerate}
\end{proof}

\begin{proof}[Proof of Corollary \ref{cor:Delta2}]
It is enough to observe that 
$$
B^{-1}\begin{pmatrix}2\\1\\0\end{pmatrix}=\begin{pmatrix}2\\-1\\0\end{pmatrix}
\quad
B\begin{pmatrix}-2\\1\\0\end{pmatrix}=\begin{pmatrix}-2\\-1\\0\end{pmatrix}
\quad\textrm{and}\quad
J\begin{pmatrix}1\\0\\-1\end{pmatrix}=\begin{pmatrix}-1\\0\\1\end{pmatrix}.
$$
\end{proof}

\section{Miscellaneous results}\label{sec:Prel}

We give a list of lemmas that will be used various times till the end of the paper. 



\begin{lem}\label{crec}
Let $x=(x_1,x_2,x_3)\in\Z^3$. For $|a|\leq 2$, if $x\in\Gamma_a\smallsetminus\Theta_a$ then the sequence $(x_1^2,x_2^2,x_3^2)$ is either strictly increasing or strictly decreasing. 
\end{lem}
\begin{proof}
Suppose that $x\in\Gamma_a\smallsetminus\Theta_a$ and write Equation \eqref{eqBa} as 
$$
(x_3^2-x_2^2)-(x_2^2-x_1^2)=a.
$$ 

\noindent\textbf{Cases $a=1$ and $a=2$.} We have $x_2\ne0$ (otherwise $x_1^2+x_3^2=a$ and $x\in\Theta_a$). If $x_1^2=x_2^2$ then $x_3^2-x_2^2=a$, which is not possible (if $a=1$, it would imply $x_2=0$). Hence $x_1^2\ne x_2^2$ and by symmetry we have $x_3^2\ne x_2^2$. 

If $x_1^2<x_2^2$ then $x_3^2-x_2^2=a+(x_2^2-x_1^2)>a>0$, hence $x_3^2>x_2^2$ and the sequence is strictly increasing. If $x_1^2>x_2^2$ then $x_3^2-x_2^2=a+(x_2^2-x_1^2)<a$, hence $x_3^2-x_2^2\leq 1$. Since $x_2\ne0$ and $x_2^2\ne x_3^2$, we deduce that $x_3^2-x_2^2<0$, which implies that the sequence is strictly decreasing.\\

\noindent\textbf{Cases $a=-1$ and $a=-2$.} We have $x_2^2\ne1$ (otherwise $x_3^2+x_1^2=a+2$ and $x\in\Theta_{a}$) and $x_2\ne0$ (otherwise we would have $x_3^2+x_1^2=a<0$). Therefore, we have $x_2^2\geq4$ and the difference between $x_2^2$ and any other square is at least $3$ or non-positive. In particular, if $x_2^2-x_1^2<-a$ then $x_2^2< x_1^2$ (note that if $x_1^2=x_2^2$ then $x_3^2-x_2^2=a$, but the difference of two squares cannot be $-2$, and if $a=-1$ then $x\in\Theta_{-1}$).

If $x_1^2>x_2^2+a$ then $x_3^2-x_2^2=a+(x_2^2-x_1^2)<0$, hence $x_3^2<x_2^2<x_1^2$. 

If $x_1^2<x_2^2+a$ then $x_3^2-x_2^2=a+(x_2^2-x_1^2)>0$, hence $x_3^2>x_2^2>x_1^2$.\\

\noindent\textbf{Case $a=0$.} Note that from Equation \eqref{eqBa} if $x_i^2=x_j^2$ for some $i\ne j$, then $x_1^2=x_2^2=x_3^2$, in which case $x\in\Theta_0$. If $x_1^2<x_2^2$ then $0=x_1^2-2x_2^2+x_3^2<x_3^2-x_2^2$ and the sequence is strictly increasing. If $x_1^2>x_2^2$ then $0=x_1^2-2x_2^2+x_3^2>x_3^2-x_2^2$ and the sequence is strictly decreasing. 
\end{proof}

\begin{lem}\label{prod}
Let $\varepsilon=\pm1$ and $x=(x_1,x_2,x_3)\in\Z^3$. Writing $(y_1,y_2,x_3)=B^{\varepsilon}x$, we have  
\begin{enumerate}
\item if $\varepsilon x_1x_2\ge 0$ then $\varepsilon y_1y_2\ge0$; 
\item if $\varepsilon x_1x_2>0$ then $\varepsilon y_1y_2>0$; and 
\item if $|x_2|>|x_1|$ and $\varepsilon x_1x_2<0$ then $\varepsilon y_1y_2>0$.
\end{enumerate}
\end{lem}
\begin{proof}
By definition of $B$, we have
$$
\begin{aligned}
\varepsilon y_1y_2&=\varepsilon(3x_1+4\varepsilon x_2)(2\varepsilon x_1+3x_2)\\
&=6x_1^2+17\varepsilon x_1x_2+12x_2^2
\end{aligned}
$$
and we can deduce Items 1 and 2. For Item 3, note that
$$
\begin{aligned}
\varepsilon y_1y_2&=6x_1^2+17\varepsilon x_1x_2+12x_2^2\\
&=6(x_1+\varepsilon x_2)^2+5\varepsilon x_1x_2+6x_2^2\\
&=6(x_1+\varepsilon x_2)^2+5x_2(\varepsilon x_1+x_2)+x_2^2\\
\end{aligned}
$$
and since $\varepsilon x_1+x_2$ has the same sign as $x_2$, we have $5x_2(\varepsilon x_1+x_2)>0$. 
\end{proof}

\begin{lem}\label{BB}
Let $\varepsilon=\pm 1$ and $|a|\leq 2$. Any strictly increasing sequence $(x_1,x_2,x_3)$ (in absolute value) in $\Gamma_a$, when multiplied by $B^{\varepsilon}$, produces a strictly decreasing sequence (in absolute value) $(y_1,y_2,x_3 )$ in $\Gamma_a$ satisfying $\varepsilon y_1y_2>0$.
\end{lem}
\begin{proof}
We first prove that $(y_1,y_2,x_3 )$ is not in $\Theta_a$. Since $(x_1,x_2,x_3)$ is strictly increasing in absolute value, we have $|x_2|\ge1$ and $|x_3|\ge2$, hence the only cases to check are when $a=0$, and when $a=2$ and $(x_1,x_2,x_3)=(0,\pm1,\pm2)$. In the latter case, we have $B^\varepsilon(0,\pm1,\pm2)=(\pm 4,\pm 3,\pm2)$ which is not in $\Theta_2$. Suppose for a contradiction that $(y_1,y_2,x_3)$ is in $\Theta_0$ (hence in particular $y_1=y_2$). Since by definition of $B$ we have $y_1=3x_1+4\varepsilon x_2$ and $y_2=2\varepsilon x_1+3x_2$, we obtain $(3-2\varepsilon)x_1=(3-4\varepsilon)x_2$, hence
$$
1<\frac{|x_2|}{|x_1|}=\frac{|3-2\varepsilon|}{|3-4\varepsilon|}\leq1
$$
which is absurd. 

Therefore, by Lemma \ref{crec}, it is enough to show that $|y_1|>|y_2|$ and $\varepsilon y_1y_2>0$. 

Suppose that $\varepsilon x_1x_2$ is non-negative. We have 
$$
|y_1|=|\varepsilon y_1|=|3\varepsilon x_1+4x_2|>|2\varepsilon x_1+3x_2|=|y_2|
$$ 
where the inequality comes from the fact that $\varepsilon x_1$ and $x_2$ have the same sign. Note also that $\varepsilon y_1y_2$ is non-negative by Lemma \ref{prod}, and since 
$$
|y_2|=|2\varepsilon x_1+3x_2|>|x_2|>0,
$$ 
we obtain $\varepsilon y_1y_2>0$. 

If $\varepsilon x_1x_2$ is negative, write $u=x_1+\varepsilon x_2$. We have 
$$
\begin{aligned}
\varepsilon uy_2&=2x_1^2+5\varepsilon x_1x_2+3x_2^2\\
&=2x_1^2+4\varepsilon x_1x_2+2x_2^2+\varepsilon x_1x_2+x_2^2\\
&=2(x_1+\varepsilon x_2)^2+\varepsilon x_1x_2+x_2^2
\end{aligned}
$$
which is positive, since by hypothesis we have $|x_2|>|x_1|$. Since $y_1=u+\varepsilon y_2$, we deduce
$$
|y_1|=|u+\varepsilon y_2|>|y_2|
$$ 
because $u$ is non-zero (by hypothesis) and because $u$ and $\varepsilon y_2$ have the same sign. Note also that $\varepsilon y_1y_2$ is positive by Lemma \ref{prod}. 
\end{proof}

\begin{lem}\label{suc}
Let $\varepsilon=\pm1$. Let $x=(x_1,x_2,x_3)\in\Z^3$ be such that $\varepsilon x_1x_2\ge 0$ and $x_1\ne0$. For each $n\ge0$, let $u_n$ and $v_n$ be defined by $(u_n,v_n,x_3)=B^{\varepsilon n}x$. For each $n\ge0$, we have
\begin{enumerate}
\item $|u_{n+1}|>|u_n|$;
\item $|v_{n+1}|>|v_n|$; 
\item $u_n\ne0$;
\item $\varepsilon u_nv_n\ge0$.
\end{enumerate} 
In particular, $v_n\ne0$ for each $n\ge1$. Moreover, if $|a|\leq2$ and $x\in\Gamma_a$ is strictly decreasing in absolute value (hence $v_n\ne0$), then $(u_n,v_n,x_3)$ is strictly decreasing in absolute value (it is false in general if $x$ does not satisfy the hypothesis $\varepsilon x_1x_2\ge 0$). 
\end{lem}
\begin{proof}
Note that the lemma is trivial for $n=0$. Suppose that the lemma holds for some integer $n\ge0$. Since $\varepsilon u_nv_n\ge0$ and $u_n\ne0$ we have
$$
|u_{n+1}|=|3u_n+4\varepsilon v_n|>|u_n|
$$
and 
$$
|v_{n+1}|=|2\varepsilon u_n+3v_n|> |v_n|
$$ 
(where the equalities come from the definition of $B$). Hence also $u_{n+1}\ne0$ and 
$$
\begin{aligned}
\varepsilon u_{n+1}v_{n+1}&=\varepsilon(3u_n+4\varepsilon v_n)(2\varepsilon u_n+3v_n)\\
&=6u_n^2+12v_n^2+17\varepsilon u_nv_n
\end{aligned}
$$
is non-negative. 

We now prove the last statement of the lemma. If $n=0$ there is nothing to prove, so we assume $n\ge1$. By Lemma \ref{crec}, it is enough to prove that $(u_n,v_n,x_3)$ is not in $\Theta_a$ (the point is that $x_3$ does not change as $n$ varies and $|x_3|$ remains the minimum of the sequence of absolute values). 

Since $n\geq1$, we have both $u_n\ne0$ and $v_n\ne0$. Hence the only possibilities for $(u_n,v_n,x_3)$ to be in $\Theta_a$ are when $a=-1$ and $(u_n,v_n,x_3)=(\pm1,\pm1,0)$, or $a=0$, or $a=2$ and $(u_n,v_n,x_3)=(\pm2,\pm1,0)$. By Item 2, if $v_n=\pm1$ then $n=1$. We have $B^\varepsilon(x_1,x_2,x_3)=(3x_1+4\varepsilon x_2,2\varepsilon x_1+3x_2,x_3)$. When $a=2$, this leads to $3x_1+4\varepsilon x_2=\pm2$, which is impossible since $3x_1$ and $4\varepsilon x_2$ are of the same sign by hypothesis. An analogous argument discards the case $a=-1$. If $a=0$ then by Item 1 we have $|u_n|\geq|u_1|>|u_0|=|x_1|>|x_3|$ since the initial sequence is supposed to be strictly decreasing. 
\end{proof}

\begin{lem}\label{lempres}
If 
$$
w=B^{n_k}J\dots B^{n_1}J
$$ 
is an element of $H$, where $k\ge 1$ and each $n_i$ is a non-zero integer, then its third column is strictly decreasing in absolute value and the entry $w_{23}$ of the matrix $w$ at line $2$ and column $3$ is distinct from $0$.
\end{lem}
\begin{proof}
We prove by induction on the right subwords 
$$
W^{s}=B^{n_s}J\dots B^{n_1}J^{r}
$$ 
of $w$ that the third column of each $W^s$ is strictly decreasing in absolute value and that the entry $W^{s}_{23}$ of the matrix $W^{s}$ at line $2$ and column $3$ is distinct from $0$. 

Suppose that $s=1$. Let $\varepsilon$ be $1$ if $n_1$ is positive and $-1$ otherwise. By Lemma \ref{suc}, taking for $x$ the third column of the matrix $J$, we need only prove that the third column of $W^1$ is strictly decreasing in absolute value. Let $u_n$ and $v_n$ be like in Lemma \ref{suc}. Since $n_1\ge1$, we have $v_{n_1}\ne0$ (by Lemma \ref{suc}), hence the third column $(u_{n_1},v_{n_1},0)$ of $W^1$ is not in $\Theta_1$ and we can apply Lemma \ref{crec}, which implies that $(u_{n_1},v_{n_1},0)$ is strictly decreasing in absolute value. 

Suppose that the property holds up to $s-1$. Hence by hypothesis of induction, the third column of $W^{s-1}$ is an element $(x_3,x_2,x_1)$ of $\Gamma_1$, such that $x_2\ne0$ and which is strictly decreasing in absolute value. When multiplied by $J$, it becomes a strictly increasing sequence (in absolute value) $(x_1,x_2,x_3)$. Therefore, by Lemma \ref{BB}, when the latter is multiplied by $B^\varepsilon$, it gives a strictly decreasing  (in absolute value) sequence $(y_1,y_2,x_3)$ in $\Gamma_1$ such that $\varepsilon y_1y_2$ is positive. By Lemma \ref{suc}, taking for $x$ the third column $(y_1,y_2,x_3)$ of the matrix $B^\varepsilon JW^{s-1}$, we need only prove that the third column of $W^{s}$ is strictly decreasing in absolute value. If $n_s=\pm1$, then we have nothing more to prove. If $n_s\ge2$, letting $u_n$ and $v_n$ be like in Lemma \ref{suc}, we have $v_{n_s-1}\ne0$ and we can conclude that $(u_{n_s-1},v_{n_s-1},x_3)$ is strictly decreasing in absolute value. 
\end{proof}

We finish this section by a folklore Lemma. 

\begin{lem}\label{lem:distmin}
If $y=(y_1,y_2,\dots,y_N)$ is a non-trivial B\"uchi sequence of length $N$ which is increasing in absolute value then, for each index $n\geq2$, we have 
\begin{equation}\label{eq:distmin1}
|y_{n+1}|-|y_n|<|y_n|-|y_{n-1}|. 
\end{equation}
\end{lem}
\begin{proof}
If for some index $n\geq2$ we have $|y_{n+1}|-|y_n|\geq |y_n|-|y_{n-1}|$ then $|y_{n+1}|\geq 2|y_n|-|y_{n-1}|$, hence 
$$
2-y_{n-1}^2+2y_n^2=y_{n+1}^2\geq 4y_{n^2}-4|y_n y_{n-1}|+y_{n-1}^2
$$
and we get 
$$
2\geq 2y_n^2-4|y_n y_{n-1}|+2y_{n-1}^2
$$
hence
$$
1\geq (|y_n|-|y_{n-1}|)^2
$$
which implies that the sequence is trivial. 
\end{proof}

\section{Presentation of the group $H$}\label{sec:pres}

Theorem \ref{pres} is an easy corollary of Lemma \ref{lempres}. We consider an arbitrary element of $H$
$$
w=J^{\ell}B^{n_k}J\dots B^{n_2}JB^{n_1}J^{r}
$$
where $k\ge 1$, each $n_i$ is a non-zero integer, and $\ell$ and $r$ are $0$ or $1$. We will prove that $w$ is \emph{not} the identity matrix and the theorem will follow (since the only non-empty word that we are missing is $J$ which is distinct from $I$).

Note that if $\ell=1$ then it is enough to show that $JwJ$ is not the identity, and if $\ell=r=0$ then it is enough to show that $B^{n_k}wB^{-n_k}$ is not the identity. So, without loss of generality, we can assume $\ell=0$ and $r=1$, and conclude with Lemma \ref{lempres}.

\section{Proof of Theorem \ref{main2}}\label{sec:main2}

By Corollary \ref{cor:Delta2}, we need only prove the \emph{unicity} part of the theorem. 

\begin{defin}\label{length}
If $M=J^{\ell}B^{n_k}J\dots B^{n_2}JB^{n_1}J^{r}$, where $k\ge1$, each $n_i$ is a non-zero integer, and $\ell$ and $r$ are $0$ or $1$, then we will call $k$ the \emph{length of $M$}. Elements of $H$ of \emph{length $0$} are $I$ and $J$. We will refer to $(\ell,n_k,\dots,n_1,r)$ as to the \emph{sequence of powers associated to $M$}. 
\end{defin}

Next Lemma is a corollary of Lemma \ref{lempres} which we already used in order to find the presentation of $H$.

\begin{lem}\label{m23}
If $M\in H$ is such that $M_{23}=0$ then either $M=I$ or $M=J$ or $M=B^n$ or $M=JB^n$ for some $n\in\Z$.
\end{lem}
\begin{proof}
Suppose that $M$ is neither $I$, nor $J$, and neither of the form $B^n$ nor $JB^n$. Hence in particular $M$ has length at least $1$ and can be written as 
$$
M=J^{\ell}B^{n_k}J\dots B^{n_2}JB^{n_1}J^{r}
$$
for some $k\ge1$ and where each $n_i$ is a non-zero integer, and $\ell$ and $r$ are $0$ or $1$. We want to prove that $M_{23}$ is non-zero. 

If $\ell=0$ and $r=1$ then we conclude by Lemma \ref{lempres}. Also if $\ell=1$ and $r=1$ then $(JM)_{23}$ is non-zero by Lemma \ref{lempres}, hence $M_{23}$ is non-zero. So we may suppose that $r=0$. 

Since the only words of length $1$ with $r=0$ are of the form $B^n$ or $JB^n$, we may suppose that the length of $M$ is at least $2$. Let $M_0\in H$ be such that $M=M_0B^{n_2}JB^{n_1}$. By Lemma \ref{lempres}, we have $(M_0B^{n_2}J)_{23}\ne0$. We can conclude that $M_{23}$ is non-zero because multiplying by $B$ on the right does not affect the third column. 
\end{proof}

Next lemma resumes some basic properties of the matrix $B$.

\begin{lem}\label{BnDiag}
The characteristic polynomial of $B$ is $x^3-7x^2+7x-1$, its eigen values are $2\sqrt{2}+3$, $-2\sqrt{2}+3$ and $1$, and 
$$
\begin{pmatrix}
\sqrt{2}&-\sqrt{2}&0\\
1&1&0\\
0&0&1
\end{pmatrix}
$$
is a matrix of eigen vectors. Hence for any $n\in\Z$ we have
$$
B^n=\frac{1}{2\sqrt{2}}\begin{pmatrix}
\sqrt{2}(\bar\alpha^n+\alpha^n)&2(-\bar\alpha^n+\alpha^n)&0\\
-\bar\alpha^n+\alpha^n&\sqrt{2}(\bar\alpha^n+\alpha^n)&0\\
0&0&2\sqrt{2}
\end{pmatrix}
$$
where $\alpha=2\sqrt 2+3$ and $\bar\alpha=\alpha^{-1}$ is the conjugate of $\alpha$ in $\Z[\sqrt 2]$. Moreover, each entry $(i,j)$ in $B^n$, with $i,j\in\{1,2\}$, satisfies the recurrence relation $B^{n}_{i,j}=6B^{n-1}_{i,j}-B^{n-2}_{i,j}$ (the initial values are given by the identity matrix and $B$ at the corresponding entry).
\end{lem}

We believe that the recurrence relation described above could be very useful to solve Problems A and B (see Section \ref{sec:strat}). For the purposes of this section, we will only need the following:

\begin{cor}\label{cor:Bn}
The matrices $B^n$ and $JB^n$, for $n\in\Z\smallsetminus\{0\}$, have second row distinct from $(0,\pm1,0)$, and the diagonal entries are positive integers.
\end{cor}
\begin{proof}
Observe that both $\alpha$ and $\bar\alpha$ are positive real numbers. 
\end{proof}

\begin{defin} 
A sequence in $\Gamma_2$ is \emph{odd} if it is in the orbit of one of 
$(\pm 1,0,\pm 1)$ and it is \emph{even} if it is in the orbit of one of $(\pm 2,1,0)$.
\end{defin}

\begin{lem}\label{lem:evod}
If a sequence $(x_1,x_2,x_3)\in\Gamma_2$ is odd then $x_1$ and $x_3$
are odd, and $x_2$ is even. If it is even, then $x_1$ and $x_3$ are even and
$x_1$ is odd.
\end{lem}
\begin{proof} 
If $x_1$ and $x_3$ are odd and $x_2$ is even, then $3x_1+4x_2$ is odd and $2x_1+3x_2$ is even, hence any odd sequence in $\Gamma_2$ satisfies the desired property. The case of even sequences is done similarly.
\end{proof}

Next lemma finishes the proof of the theorem. 

\begin{lem}\label{unic1}
Let $M\in H$ and $\delta,\delta'\in\Delta_2$ such that $M\delta=\delta'$. If $\delta$ is odd then $M$ is either $I$ or $J$ (in the latter case, $\delta$ must be $(1,0,1)$ or $(-1,0,-1)$). If $\delta$ is even, then $M$ is the identity. In all cases we have $\delta=\delta'$.
\end{lem}
\begin{proof}
Write $M=(m_{ij})$ and suppose first that $\delta$ is odd, i.e. $\delta=(\varepsilon_1,0,\varepsilon_2)$ for some $\varepsilon_1,\varepsilon_2\in\{\pm1\}$. Since $M\delta=\delta'$, we have $\varepsilon_1 m_{21}+\varepsilon_2m_{23}=0$. Since the second row of $M$ is in $\Omega_{1}$ (see Theorem \ref{act}), we have 
$$
-2m_{21}^2+m_{22}^2-2m_{23}^2=1
$$
hence 
$$
(m_{22}-2m_{21})(m_{22}+2m_{21})=m_{22}^2-4m_{21}^2=1.
$$ 
We have then $m_{22}-2m_{21}=m_{22}+2m_{21}$, hence  $m_{21}=0$ and $m_{22}^2=1$. So the second row of $M$ is $(0,\pm1,0)$ and we conclude by Lemma \ref{m23} and Corollary \ref{cor:Bn}.


Suppose now that $\delta$ is even, i.e. $\delta=(2\varepsilon,1,0)$ for some $\varepsilon\in\{\pm1\}$. Since $\delta'$ is in the orbit of $\delta$, it is also even by Lemma \ref{lem:evod}, so $\delta'=(2\varepsilon',1,0)$ for some $\varepsilon'\in\{\pm1\}$. Since $M\delta=\delta'$, we have $2\varepsilon m_{31}+m_{32}=0$. Since the third row is in $\Omega_{-2}$ (see Theorem \ref{act}), we have 
$$
-2m_{31}^2+m_{32}^2-2m_{33}^2=-2
$$
hence $2m_{31}^2-2m_{33}^2=-2$, which implies $m_{31}=m_{32}=0$ and $m_{33}=\pm1$. Since the third column is in $\Gamma_1$, we have $m_{13}^2-2m_{23}^2+m_{33}^2=1$, hence $m_{13}^2-2m_{23}^2=0$, which implies $m_{13}=m_{23}=0$. By Lemma \ref{m23}, the only possibilities for $M$ are $I$, $B^n$ or $JB^n$ for some $n\in\Z$. Hence in particular, we can assume that all $m_{ii}$ are positive by Corollary \ref{cor:Bn} (hence $m_{33}=1$). 


On the other hand, we also have $2\varepsilon m_{11}+m_{12}=2\varepsilon'$. Since the second row is in $\Omega_{1}$, we have $-2m_{21}^2+m_{22}^2-2m_{23}^2=1$, hence $-2m_{21}^2+m_{22}^2=1$, hence
$$
-(1-m_{22})^2+2m_{22}^2=2
$$
and we finally obtain two solutions for $m_{22}$, which are $1$, in which case $M=I$; or $-3$,
which is impossible.
\end{proof}

\section{Congruences modulo $8$}\label{sec:orb}

Next theorem shows that in order to know in which orbit a length $3$ B\"uchi sequence is, it is enough to consider the sequence modulo $8$ (`congruent' means `congruent modulo $8$' in this section). 

\begin{thm}\label{thm:orb}
A B\"uchi sequence $x=(x_1,x_2,x_3)$ is in the orbit of: 
\begin{enumerate}
\item $(1,0,1)$ if and only if both $x_1$ and $x_3$ are congruent to $1$ or $3$;
\item $(-1,0,-1)$ if and only if both $x_1$ and $x_3$ are congruent to $-1$ or $-3$;
\item $(-1,0,1)$ if and only if $x_1$ is congruent to $-1$ or $-3$ and $x_3$ is congruent to $1$ or $3$, or $x_3$ is congruent to $-1$ or $-3$ and $x_1$ is congruent to $1$ or $3$;
\item $(2,1,0)$ if and only if either $x_1$ or $x_3$ is congruent to $2$; and
\item $(-2,1,0)$ if and only if either $x_1$ or $x_3$ is congruent to $-2$.
\end{enumerate}
\end{thm}
\begin{proof}
Recall that
$$
Bx=\begin{pmatrix}3x_1+4x_2\\2x_1+3x_2\\x_3\end{pmatrix}\qquad\textrm{and}\qquad
B^{-1}x=\begin{pmatrix}3x_1-4x_2\\-2x_1+3x_2\\x_3\end{pmatrix}.
$$ 

Suppose first that $x$ is an odd sequence. Since $x_2$ is even (see Lemma \ref{lem:evod}), $3x_1\pm 4x_2$ is congruent to $3x_1$. Hence, if $x_1$ is congruent to $1$ or $3$ then $3x_1\pm 4x_2$ is also congruent to $1$ or $3$. Similarly, if $x_1$ is congruent to $-1$ or $-3$ then $3x_1\pm 4x_2$ is also congruent to $-1$ or $-3$. From this observations and the fact that multiplying by $J$ interchanges $x_1$ and $x_3$, it is easy to conclude for Items 1, 2 and 3 of the Theorem. 

If $x$ is an even sequence then $x_2$ is odd and $3x_1+4x_2$ is congruent to $3x_1+4$. So if $x_1$ is congruent to $2$ then also $3x_1+4x_2$ is congruent to $2$, and if $x_1$ is congruent to $-2$ then also $3x_1+4x_2$ is congruent to $-2$. This allows us to conclude for Items 4 and 5. 
\end{proof}

Next Lemma says that B\"uchi's problem has a positive answer for $\Z/8\Z$ (Hensley \cite{Hensley2} solved B\"uchi's problem modulo any power of a prime, but did not try to find optimal lower bounds for the length of non-trivial sequences). 

\begin{lem}\label{lem:8}
Modulo $8$, all B\"uchi sequences of length $3$ are trivial.
\end{lem}
\begin{proof}
Let $x=(x_1,x_2,x_3)$ be a B\"uchi sequence modulo $8$. Squares are $0$, $1$ and $4$. If $x_1^2=0$ then $-2x_2^2+x_3^2=2$, hence $x_2^2=1$ and $x_3^2=4$. Therefore, the sequence $(x_1^2,x_2^2,x_3^2)$ is a sequence of consecutive squares, which implies that $x$ is trivial. If $x_1^2=1$ then $-2x_2^2+x_3^2=1$, hence $x_2^2=0$ or $x_2^2=4$. If $x_2^2=0$ then $x_3^2=1$ and we obtain a sequence of consecutive squares. If $x_2^2=4$ then $x_3^2=1$, but again the sequence $(1,4,1)=(1^2,2^2,3^2)$ is a sequence of consecutive squares.
\end{proof}

\begin{rem}\label{rem:8}
If $x=(x_1,x_2,x_3)$ is an even B\"uchi sequence and if for example $x_1$ is congruent to $\pm2$, then by Lemma \ref{lem:8} its sequence of squares is either of the form $(2^2,3^2,4^2)$ or $(2^2,1^2,0^2)$, hence $x_3$ is congruent to $0$ or $4$. Unfortunately, this argument does not give any information for odd sequences. 
\end{rem}

Next corollaries are the key points of our strategy to solve B\"uchi's problem (see Section \ref{sec:strat}). 

\begin{cor}\label{cor:sameorb}
Given a length $5$ B\"uchi sequence $(x_1,\dots,x_5)$, after changing the signs of $x_1$, $x_3$ or $x_5$ if necessary, $(x_1,x_2,x_3)$ and $(x_3,x_4,x_5)$ are both in the orbit of 
\begin{enumerate}
\item $(-1,0,1)$ if $x_1$ is odd; and
\item $(2,1,0)$ if $x_1$ is even.
\end{enumerate}
\end{cor}
\begin{proof}
It is immediate from Theorem \ref{thm:orb}. 
\end{proof}

Before stating next corollary, let us first introduce two definitions. 

\begin{defin}\label{def:evodd}
A B\"uchi sequence $(x_1,\dots,x_M)$ is \emph{odd} if $(x_1,x_2,x_3)$ is odd and it is \emph{even} if not. 
\end{defin}

\begin{defin}\label{def:cano}
We will call a length $5$ sequence of integers $x=(x_1,\dots,x_5)$ \emph{canonical} if it satisfies
\begin{enumerate}
\item $x_1$ and $x_5$ are congruent to $2$; and
\item either $x_{4}$ is congruent to $1$ or $-3$, and $x_{2}$ is congruent to $-1$ or $3$, or $x_{4}$ is congruent to $-1$ or $3$, and $x_{2}$ is congruent to $1$ or $-3$.
\end{enumerate}
\end{defin}

Note that in the definition above we do not require the sequence to be a B\"uchi sequence. 

\begin{cor}\label{cor:sameorb2}
Given a length $8$ B\"uchi sequence $y=(y_1,\dots,y_8)$, after changing the signs of the $y_i$ if necessary, there exists $1\leq j\leq 4$ such that $(y_j,\dots,y_{j+4})$ is canonical.
\end{cor}
\begin{proof}
Let $z=(z_1,\dots,z_7)$ be the (unique) even length $7$ subsequence of $y$. Let $k\in\{1,3\}$ be such that $z_k$ is congruent to $\pm2$ (such a $k$ exists by Theorem \ref{thm:orb}). Write $x=(x_1,\dots,x_5)=(z_k,\dots,z_{k+4})$ (so the index $j$ of the statement can be chosen to be $k$ if $z_1=y_1$ and $k+1$ if $z_1=y_2$). 

Since $x_1=z_{k}$ is congruent to $\pm2$, by Remark \ref{rem:8}, $x_3$ is congruent to $0$ or $4$, and by Theorem \ref{thm:orb}, $x_5$ is congruent to $\pm2$. Also by Theorem \ref{thm:orb}, both $x_2$ and $x_4$ are congruent to $\pm 1$ or $\pm 3$. So we can obtain the desired sequence by multiplying $x_1$, $x_2$ and $x_5$ by $-1$ if necessary. 
\end{proof}

\begin{cor}\label{cor:cantriv}
If all canonical B\"uchi sequences are trivial then all length $8$ B\"uchi sequences are trivial. 
\end{cor}
\begin{proof}
Let $y$ be a length $8$ B\"uchi sequence and $x$ be a canonical subsequence of $y$ (it exists by Corollary \ref{cor:sameorb2}). Since $x$ is trivial by hypothesis, also $y$ is trivial (indeed it is easy to see that if there are two consecutive terms $x_i$ and $x_{i+1}$ in a B\"uchi sequence such that $|x_i|=|x_{i+1}|\pm1$ then the sequence is trivial). 
\end{proof}

\section{A strategy for B\"uchi's Problem}\label{sec:strat}

Let $x=(x_1,\dots,x_5)$ be a length $5$ B\"uchi sequence. By changing the signs of $x_1$, $x_3$ or $x_5$ if necessary, we may suppose that $(x_1,x_2,x_3)$ and $(x_3,x_4,x_5)$ are both in the orbit of $(2,1,0)$, or both in the orbit of $(-1,0,1)$ (see Corollary \ref{cor:sameorb}). By Theorem \ref{main2}, we know that there exist unique matrices $M_1$, $M_2$ and $M_3$ and unique $\delta,\delta'\in\Delta_2$ such that 
\begin{equation}\label{eq:Mstrat}
\begin{cases}
(x_1,x_2,x_3)=M_1\delta\\
(x_2,x_3,x_4)=M_2\delta'\\
(x_3,x_4,x_5)=M_3\delta
\end{cases}
\end{equation}
and if we write $M_x=JM_3M_1^{-1}$ then we have 
\begin{equation}\label{eq:Mstrat2}
M_x\begin{pmatrix}x_1\\x_2\\x_3\end{pmatrix}=\begin{pmatrix}x_5\\x_4\\x_3\end{pmatrix}.
\end{equation}
Note that the matrix $M_x$ is uniquely determined by $x$ \emph{once the signs of the $x_i$ have been chosen}.


\begin{lem}\label{lem:BB-1}
If $M_x=B$ or $B^{-1}$ then the sequence $x$ is trivial. 
\end{lem}
\begin{proof}
If $M_x=B$ or $B^{-1}$ then we have $2x_1\pm 3x_2=x_4$, hence 
$$
2=x_4^2-2x_3^2+x_2^2=(2x_1\pm 3x_2)^2-2x_3^2+x_2^2=4x_1^2\pm12x_1x_2+10x_2^2-2x_3^2
$$
and since $x_3^2-2x_2^2+x_1^2=2$, this gives
$$
2=4x_1^2\pm12x_1x_2+10x_2^2-2(2-x_1^2+2x_2^2)=6x_1^2\pm12x_1x_2+6x_2^2-4
$$
hence $x_1^2\pm 2x_1x_2+x_2^2=1$, which implies that $x_1\pm x_2=\varepsilon$, for some $\varepsilon\in\{-1,1\}$. Writing $\nu=-\varepsilon x_1$, one conclude easily that for each $i$ we have $x_i^2=(\nu+i-1)^2$, so the sequence $x$ is a trivial B\"uchi sequence. 
\end{proof}

We may write $\xi_1=(x_1,x_2,x_3)$ and $\xi_2=(x_5,x_4,x_3)$, so that we have 
$$
M_x\xi_1=\xi_2.
$$
Also for any sequence $y$, we will denote by $|y|$ the sequence of its absolute values. 

In order to prove that there is no non-trivial B\"uchi sequence of length $5$, one strategy is to try to solve the following problem by induction on $n$.\\

\noindent\textbf{Problem A.} Is it true that for all $n\geq0$, sequences $x$ whose matrix $M_x$ has length $n$ are trivial?\\

Next Lemma shows that Problem A has a positive answer for $n\leq1$. 

\begin{thm}\label{thm:BB-3}
If $M_x$ has length $\leq1$ then $x$ is trivial. 
\end{thm}
\begin{proof}
We will assume that $x$ is non-trivial and obtain a contradiction when $M_x$ has length $0$ or $1$. 

By Lemma \ref{crec}, since $x$ is non-trivial, the sequence $|x|$ is either strictly increasing or strictly decreasing. Suppose first that it is strictly increasing. 

If the length of $M$ were $0$ then we would have $M=I$ or $J$, hence $x_1=x_5$ or, respectively, $x_1=x_3$, which would give a contradiction in both cases. 

For the sake of contradiction, assume that the length of $M$ is $1$, so that $M$ has one of the following forms: $B^{\varepsilon n}$, $JB^{\varepsilon n}$, $JB^{\varepsilon n}J$ or $B^{\varepsilon n}J$, where $n\ge1$ and $\varepsilon=\pm1$. 

\emph{Case $M=B^{\varepsilon n}J$.} We have $(x_5,x_4,x_3)=B^{\varepsilon n}J\xi_1=B^{\varepsilon n}(x_3,x_2,x_1)$, hence $x_1=x_3$, which is impossible. 

\emph{Case $M=JB^{\varepsilon n}J$.} We have $(x_3,x_4,x_5)=J\xi_2=B^{\varepsilon n}J\xi_1=B^{\varepsilon n}(x_3,x_2,x_1)$, hence $x_1=x_5$, which is impossible. 

\emph{Case $M=JB^{\varepsilon n}$.} Since $|\xi_1|$ is strictly increasing, the sequence $(y_1,y_2,x_3)$ defined by $B^{\varepsilon}\xi_1$ is strictly decreasing in absolute value and satisfies $\varepsilon y_1y_2>0$ (see Lemma \ref{BB}). By Lemma \ref{suc}, $B^{\varepsilon (n-1)}B^{\varepsilon}\xi_1=J\xi_2=(x_3,x_4,x_5)$ is strictly decreasing in absolute value, which is impossible.

\emph{Case $M=B^{\varepsilon n}$.} Since $x$ is assumed to be non-trivial, we have $n>1$ by Lemma \ref{lem:BB-1}. We first prove that if $(y_1,y_2,x_3)$ is defined by $B^{\varepsilon}\xi_1$ then $|y_1|>|x_5|$. We have  
$$
\begin{aligned}
|y_1|&=|3x_1+4\varepsilon x_2|\geq 4|x_2|-3|x_1|=3(|x_2|-|x_1|)+|x_2|\\
&>3(|x_3|-|x_2|)+|x_2|=2(|x_3|-|x_2|)+|x_3|\\
&>2(|x_4|-|x_3|)+|x_3|=(|x_4|-|x_3|)+|x_4|\\
&>(|x_5|-|x_4|)+|x_4|=|x_5|
\end{aligned}
$$ 
where the strict inequalities come from Lemma \ref{lem:distmin}. By Lemma \ref{BB}, the sequence $(y_1,y_2,x_3)$ is strictly decreasing in absolute value and satisfies $\varepsilon y_1y_2>0$, hence applying Lemma \ref{suc} repeatedly $(n-1)$ times, the sequence 
$$
(x_5,x_4,x_3)=M\xi_1=B^{\varepsilon (n-1)}B\xi_1
$$ 
satisfies $|x_5|>|x_5|$, which is absurd. So the lemma is proven for $x$ strictly increasing in absolute value. 

Suppose now that $|x|$ is strictly decreasing and consider $\bar x=(x_5,\dots,x_1)$. There exists a unique matrix $M_{\bar x}$ such that $M_{\bar x}(x_5,x_4,x_3)=(x_1,x_2,x_3)$, hence $M_{\bar x}^{-1}(x_1,x_2,x_3)=(x_5,x_4,x_3)$. Therefore, we have $M_{\bar x}^{-1}=M_x$ and since $|\bar x|$ is strictly increasing, we know from the study above that $M_{\bar x}$, hence also $M_{\bar x}^{-1}=M_x$, cannot have length $\leq 1$ if $x$ is non-trivial. 
\end{proof}

\begin{rem}
Suppose that we want to prove that there is no non-trivial B\"uchi sequences of length $6$. Since in a B\"uchi sequence of length $6$, there is exactly one odd subsequence of length $5$ and one even subsequence of length $5$ (see Definition \ref{def:evodd}), it is enough to show that there is no odd sequence of length $5$ or that there is no even sequence of length $5$. Therefore, it would be enough to solve Problem A for $n\geq2$ and assuming, for example, that $x$ is in the orbit of $(2,1,0)$. 
\end{rem}

We will finish this section by presenting a strategy to try to prove that all B\"uchi sequences of length $8$ are trivial. 

The reciprocal of Lemma \ref{lem:BB-1} is not true in general. Indeed, there are counter-examples for both odd and even sequences. For example with $x=(-1,2,3,-4,5)$, we have $\delta=(-1,0,1)$, $M_1=JBJ$ and $M_3=JB^{-1}JB^{-1}J$, hence $M_x=B^{-1}JB^{-2}J$. With $x=(2,3,4,-5,-6)$, we have $\delta=(2,1,0)$, $M_1=JBJ$ and $M_3=JBJB^{-1}JB^{-1}$, hence $M_x=BJB^{-1}JB^{-1}JB^{-1}J$. 

\begin{lem}\label{lem:BB-2}
Assume that $x$ is canonical (as defined in \ref{def:cano}). If $x$ is trivial then $M_x=B$ or $B^{-1}$. 
\end{lem}
\begin{proof}
Since $x$ is trivial, there exists an integer $n\in\Z$ such that $x_i=\varepsilon_i(n+i)$, where $\varepsilon_i\in\{-1,1\}$  for each $i=1,\dots,5$. Writing $x_1=8m+2$, we have 
$$
n=\varepsilon_1(8m+2)-1
$$
hence 
$$
x_5=\varepsilon_5(\varepsilon_1(8m+2)+4)
$$
and since $x_5$ is by hypothesis congruent to $2$ modulo $8$, we have $\varepsilon_5\varepsilon_1=-1$. Also we have 
$$
x_4=\varepsilon_4(\varepsilon_1(8m+2)+3)\qquad\textrm{and}\qquad
x_2=\varepsilon_2(\varepsilon_1(8m+2)+1)
$$
hence 
\begin{itemize}
\item $x_4$ is congruent to $1$ or $-3$ if and only if $\varepsilon_4=1$; and 
\item $x_2$ is congruent to $-1$ or $3$ if and only if $\varepsilon_2=1$.
\end{itemize}
Since the sequence is canonical, we have $\varepsilon_2=\varepsilon_4$. 

Writing $\varepsilon=-\varepsilon_1\varepsilon_2$, we have 
$$
B^{\varepsilon}\begin{pmatrix}x_1\\x_2\\x_3\end{pmatrix}=
B^{\varepsilon}\begin{pmatrix}\varepsilon_1(n+1)\\\varepsilon_2(n+2)\\x_3\end{pmatrix}=
\begin{pmatrix}3\varepsilon_1(n+1)+4\varepsilon\varepsilon_2(n+2)\\2\varepsilon\varepsilon_1(n+1)+3\varepsilon_2(n+2)\\x_3\end{pmatrix}
$$
hence 
$$
B^{\varepsilon}\begin{pmatrix}x_1\\x_2\\x_3\end{pmatrix}=
\begin{pmatrix}\varepsilon_1(3(n+1)-4(n+2))\\\varepsilon_2(-2(n+1)+3(n+2))\\x_3\end{pmatrix}=
\begin{pmatrix}\varepsilon_1(-n-5)\\\varepsilon_2(n+4)\\x_3\end{pmatrix}
$$
and we can conclude since $\varepsilon_1=-\varepsilon_5$ and $\varepsilon_2=\varepsilon_4$.
\end{proof}

\noindent\textbf{Problem B.} Let $x=(x_1,\dots,x_5)$ be a canonical sequence. Suppose that there exist matrices $M_1$, $M_2$ and $M_3$ in $H$ such that $(x_1,x_2,x_3)=M_1(2,1,0)$, $(x_2,x_3,x_4)=M_2\delta$ and $(x_3,x_4,x_5)=M_3(2,1,0)$, where $\delta$ is $(\pm1,0,\pm1)$. Is it the case that $M_x=JM_3M_1^{-1}$ is necessarily either $B$ or $B^{-1}$. 

\begin{thm}\label{thm:PB}
If Problem B has a positive answer then there are no non-trivial B\"uchi sequences of length $8$. If there are no non-trivial B\"uchi sequences of length $5$ then Problem B has a positive answer. 
\end{thm}
\begin{proof}
Suppose that Problem B has a positive answer and let $y$ be a B\"uchi sequence of length $8$. By Corollary \ref{cor:cantriv} there exists a canonical B\"uchi subsequence $x$ of $y$. By Theorem \ref{main2} there exist matrices $M_1$, $M_2$ and $M_3$ in $H$ satisfying the hypothesis of Problem B. Hence $M_x$ is either $B$ or $B^{-1}$. By Lemma \ref{lem:BB-1}, $x$ is a trivial sequence, hence also $y$ is a trivial sequence. 

Suppose that there are no non-trivial B\"uchi sequences of length $5$. In particular, there are no non-trivial canonical B\"uchi sequences. Hence all canonical B\"uchi sequences are trivial. By Lemma \ref{lem:BB-2}, this implies that all canonical B\"uchi sequences $x$ are such that $M_x$ is $B$ or $B^{-1}$, and Problem B has a positive answer. 
\end{proof}

\subsection*{Acknowledgements}

We thank J. Browkin, H. Pasten and T. Pheidas for several discussions during the preparation of this work, and the referee for his useful comments.

\end{document}